\documentclass[12pt]{amsart}
\usepackage{amssymb}
\usepackage{amsxtra}
\usepackage{graphics}
\usepackage{latexsym}
\usepackage{enumerate}
\usepackage{xcolor}
\usepackage{amsmath}
\usepackage{amssymb,amsthm,amsfonts}
\usepackage{mathtools}
\usepackage{amscd}
\usepackage[arrow, matrix, curve]{xy}
\usepackage[mathscr]{euscript}
\usepackage{syntonly}
\input xy 
\xyoption{all} 
\title[Rigid analytic $p$-adic Simpson correspondence for line bundles ]{ Rigid analytic $p$-adic Simpson correspondence for line bundles}

\author[Ziyan Song]{Ziyan Song}
\email{ziyans@mail.ustc.edu.cn}
 
\address{School of Mathematical Sciences,
University of Science and Technology of China, Hefei, 230026, China}

\begin{document}
\theoremstyle{plain}
\newtheorem{thm}{Theorem}[subsection]
\newtheorem{theorem}[thm]{Theorem}
\newtheorem*{theorem*}{Theorem}
\newtheorem*{definition*}{Definition}
\newtheorem{lemma}[thm]{Lemma}
\newtheorem{sublemma}[thm]{Sublemma}
\newtheorem{corollary}[thm]{Corollary}
\newtheorem*{corollary*}{Corollary}
\newtheorem{proposition}[thm]{Proposition}
\newtheorem{addendum}[thm]{Addendum}
\newtheorem{variant}[thm]{Variant}
\newtheorem{conjecture}[thm]{Conjecture}
\theoremstyle{definition}
\newtheorem{construction}[thm]{Construction}
\newtheorem{notations}[thm]{Notations}
\newtheorem{question}[thm]{Question}
\newtheorem{problem}[thm]{Problem}
\newtheorem{remark}[thm]{Remark}
\newtheorem{remarks}[thm]{Remarks}
\newtheorem{definition}[thm]{Definition}
\newtheorem{claim}[thm]{Claim}
\newtheorem{assumption}[thm]{Assumption}
\newtheorem{assumptions}[thm]{Assumptions}
\newtheorem{properties}[thm]{Properties}
\newtheorem{example}[thm]{Example}
\numberwithin{equation}{subsection}

\newcommand{\sA}{{\mathcal A}}
\newcommand{\sB}{{\mathcal B}}
\newcommand{\sC}{{\mathcal C}}
\newcommand{\sD}{{\mathcal D}}
\newcommand{\sE}{{\mathcal E}}
\newcommand{\sF}{{\mathcal F}}
\newcommand{\sG}{{\mathcal G}}
\newcommand{\sH}{{\mathcal H}}
\newcommand{\sI}{{\mathcal I}}
\newcommand{\sJ}{{\mathcal J}}
\newcommand{\sK}{{\mathcal K}}
\newcommand{\sL}{{\mathcal L}}
\newcommand{\sM}{{\mathcal M}}
\newcommand{\sN}{{\mathcal N}}
\newcommand{\sO}{{\mathcal O}}
\newcommand{\sP}{{\mathcal P}}
\newcommand{\sQ}{{\mathcal Q}}
\newcommand{\sR}{{\mathcal R}}
\newcommand{\sS}{{\mathcal S}}
\newcommand{\sT}{{\mathcal T}}
\newcommand{\sU}{{\mathcal U}}
\newcommand{\sV}{{\mathcal V}}
\newcommand{\sW}{{\mathcal W}}
\newcommand{\sX}{{\mathcal X}}
\newcommand{\sY}{{\mathcal Y}}
\newcommand{\sZ}{{\mathcal Z}}
\newcommand{\A}{{\mathbb A}}
\newcommand{\B}{{\mathbb B}}
\newcommand{\C}{{\mathbb C}}
\newcommand{\D}{{\mathbb D}}
\newcommand{\E}{{\mathbb E}}
\newcommand{\F}{{\mathbb F}}
\newcommand{\G}{{\mathbb G}}
\newcommand{\HH}{{\mathbb H}}
\newcommand{\I}{{\mathbb I}}
\newcommand{\J}{{\mathbb J}}
\renewcommand{\L}{{\mathbb L}}
\newcommand{\M}{{\mathbb M}}
\newcommand{\N}{{\mathbb N}}
\renewcommand{\P}{{\mathbb P}}
\newcommand{\Q}{{\mathbb Q}}
\newcommand{\R}{{\mathbb R}}
\newcommand{\SSS}{{\mathbb S}}
\newcommand{\T}{{\mathbb T}}
\newcommand{\U}{{\mathbb U}}
\newcommand{\V}{{\mathbb V}}
\newcommand{\W}{{\mathbb W}}
\newcommand{\X}{{\mathbb X}}
\newcommand{\Y}{{\mathbb Y}}
\newcommand{\Z}{{\mathbb Z}}
\newcommand{\id}{{\rm id}}
\newcommand{\rank}{{\rm rank}}
\newcommand{\END}{{\mathbb E}{\rm nd}}
\newcommand{\End}{{\rm End}}
\newcommand{\Hom}{{\rm Hom}}
\newcommand{\Hg}{{\rm Hg}}
\newcommand{\tr}{{\rm tr}}
\newcommand{\Sl}{{\rm Sl}}
\newcommand{\Gl}{{\rm Gl}}
\newcommand{\Cor}{{\rm Cor}}
\newcommand{\Aut}{\mathrm{Aut}}
\newcommand{\Sym}{\mathrm{Sym}}
\newcommand{\ModuliCY}{\mathfrak{M}_{CY}}
\newcommand{\HyperCY}{\mathfrak{H}_{CY}}
\newcommand{\ModuliAR}{\mathfrak{M}_{AR}}
\newcommand{\Modulione}{\mathfrak{M}_{1,n+3}}
\newcommand{\Modulin}{\mathfrak{M}_{n,n+3}}
\newcommand{\Gal}{\mathrm{Gal}}
\newcommand{\Spec}{\mathrm{Spec}}
\newcommand{\res}{\mathrm{res}}
\newcommand{\coker}{\mathrm{coker}}
\newcommand{\Jac}{\mathrm{Jac}}
\newcommand{\HIG}{\mathrm{HIG}}
\newcommand{\MIC}{\mathrm{MIC}}

\maketitle

\begin{abstract}
The $p$-adic Simpson correspondence due to Faltings\cite{Faltings 2005} is a $p$-adic analogue of non-abelian Hodge theory. The following is the main result of this article: The correspondence for line bundles can be enhanced to a rigid analytic morphism of moduli spaces under certain smallness conditions. In the complex setting, Simpson shows that there is a complex analytic morphism from the moduli space for the vector bundles with integrable connection to the moduli space of representations of a finitely generated group as algebraic varieties. We give a $p$-adic analogue of Simpson's result.

\textbf{Keywords} Arithmetic algebraic geometry, $p$-adic Hodge theory, rigid geometry, Higgs bundles.

\textbf{Mathematical Subject Classificication} 14G22, 14H60, 14F35

\end{abstract}

\tableofcontents
\section{Introduction}
  In the complex case, Simpson correspondence generalizes the classical Narasimhan-Sashadri correspondence on Riemann surface in the case of vanishing Higgs field. To be more precise, let $X$ be a compact Kahler manifold. A Higgs bundle is a pair consisting of a holomorphic vector bundle $E$ and a holomorphic map $\theta: E \rightarrow E \otimes \Omega_{X}^{1}$ such that $\theta \wedge \theta=0$. Simpson established a one-to-one correspondence between irreducible representations of $\pi_{1}(X)$ and 
stable Higgs bundles with vanishing Chern classes (See \cite{Simpson 1992} for more details). 

  Faltings developed a $p$-adic analogue of Simpson correspondence in the case of curves. The main theorem asserts that there exists an equivalence of categories between Higgs bundles and generalized representations, if we allow $\mathbb{C}_{p}$ coefficients. Denote $\mathbb{C}_{p}=\widehat{\overline{\mathbb{Q}_{p}}}$. Let $\textbf{o}$ be the ring of integers in $\mathbb{C}_{p}$.

  In this article, we show that the $p$-adic Simpson correspondence for line bundles is rigid analytic under suitable conditions by viewing the moduli spaces of both sides as points of rigid analytic spaces. This question is the $p$-adic analogue of the fact that there is a complex analytic morphism (but not algebraic) from the moduli space for the vector bundles with integrable connection to the moduli space of representations of the fundamental group as algebraic varieties (See \cite{Simpson 1994} and \cite{Simpson}).
 
   Turning to details, the problem can be divided into two cases.
   \subsubsection{Vector bundle case}
    Let $X$ be a smooth proper curve over $\overline{\mathbb{Q}_{p}}$. In \cite{Deninger and Werner 2005a}, C. Deninger and A. Werner defined functorial isomorphisms of parallel transport along \'{e}tale paths for a class of vector bundles on $X_{\mathbb{C}_{p}}=X\times_{Spec \overline{\mathbb{Q}_{p}}} Spec \mathbb{C}_{p}$. The category of such vector bundles is denoted by $\mathcal{B}_{\mathbb{C}_{p}}^{s}$ and contains all vector bundles of degree 0 that have strongly semistable reduction. In particular, all vector bundles in $\mathcal{B}_{\mathbb{C}_{p}}^{s}$ give rise to continuous representations of the algebraic fundamental group on finite dimensional $\mathbb{C}_{p}$-vector spaces. The construction of C. Deninger and A. Werner is compatible with the construction of G. Faltings\cite{Faltings 2005} if the Higgs field $\theta$ is zero.
    
    As a special case of the line bundles, suppose that $X$ is a curve with good reduction over $\overline{\mathbb{Q}_{p}}$. Denote the Albanese variety of $X$ by $A$ and the genus of $X$ by $g$. Let $K$ be a finite extension of $\mathbb{Q}_{p}$ in $\overline{\mathbb{Q}_{p}}$ so that $A$ is defined over $K$, i.e, $A=A_{K}\otimes \overline{\mathbb{Q}_{p}}$ for any abelian variety $A_{K}$ over $K$. By \cite{Deninger and Werner 2005a} (line 30-32 of p.20) we have a continuous, $Gal(\overline{\mathbb{Q}_{p}}/K)$\footnote{In \cite{Deninger and Werner 2005a}, the notation $G_{K}$ on p.20 stands for $Gal(\overline{\mathbb{Q}_{p}}/K)$, this arises in line 11 of p.10}-equivariant homomorphism
    \[ \alpha: Pic^{0}_{X/\overline{\mathbb{Q}_{p}}}(\mathbb{C}_{p}) \longrightarrow Hom_{c}(\pi_{1}^{ab}(X), \mathbb{C}_{p}^{*}). \]
    
    Here $Hom_{c}(\pi_{1}^{ab}(X), \mathbb{C}_{p}^{*})$ is the topological group of continuous $\mathbb{C}_{p}^{*}$-valued characters of the algebraic fundamental group $\pi_{1}(X,x)$.
    
    For the rest of the article, suppose that $X/\overline{\mathbb{Q}_{p}}$ has good reduction. Our main result on line bundles is the following:
    
    \begin{theorem} (Theorem 3.2.6)
    Assume that $(Pic^{0}_{X/\overline{\mathbb{Q}_{p}}})^{an}$ is topologically $p$-torsion. Then we may enhance the set-theoretical map
    \[ \alpha: Pic^{0}_{X/\overline{\mathbb{Q}_{p}}}(\mathbb{C}_{p}) \longrightarrow Hom_{c}(\pi_{1}^{ab}(X), \mathbb{C}_{p}^{*}) \] to a rigid analytic morphism
    \[ \alpha^{an}: (Pic^{0}_{X/\overline{\mathbb{Q}_{p}}})^{an}  \longrightarrow
    (\mathbb{G}_{m}^{2g})^{an}. \]    
    \end{theorem}
    
    The main idea is to regard the morphism $\alpha^{an}$ as gluing by the rigid analytic morphism on the affinoid neighborhood from the viewpoint of the Lie algebra map. The motivation comes from the Taylor expansion and then consider the tangent space at the origin.
            
    \subsubsection{Higgs bundle case}
    For the Higgs field of rank one Higgs bundles, Faltings constructed the morphism
    \[ \Gamma(X, \Omega_{X}^{1}) \otimes \mathbb{C}_{p}(-1) \longrightarrow Hom(\pi_{1}^{ab}(X), \mathbb{C}_{p}^{*})\]
    
    The above map is the exponential of a $\mathbb{C}_{p}$-linear map into the Lie algebra of the group of representations. More precisely, the $\mathbb{C}_{p}$-linear map is 
  \[ \Gamma(X,\Omega_{X}^{1}) \otimes_{\overline{\mathbb{Q}_{p}}} \mathbb{C}_{p}(-1) \longrightarrow (\Gamma(X,\Omega_{X}^{1}) \otimes \mathbb{C}_{p}(-1)) \oplus (H^{1}(X,\mathcal{O}_{X}) \otimes \mathbb{C}_{p}) \] 
    Using the Hodge-Tate decomposition, we have the isomorphism
\[ H^{1}_{\acute{e}t}(X,\overline{\mathbb{Q}_{p}}) \otimes_{\overline{\mathbb{Q}_{p}}} \mathbb{C}_{p} \simeq  (\Gamma(X,\Omega_{X}^{1}) \otimes \mathbb{C}_{p}(-1)) \oplus (H^{1}(X,\mathcal{O}_{X}) \otimes \mathbb{C}_{p}).\]
Note that $Hom(\pi_{1}^{ab}(X),\mathbb{C}_{p}) \simeq H^{1}_{\acute{e}t}(X,\overline{\mathbb{Q}_{p}}) \otimes_{\overline{\mathbb{Q}_{p}}} \mathbb{C}_{p}$, then the exponential map from $\mathbb{C}_{p}$ to $\mathbb{C}_{p}^{*}$ induces the map
  \[ Hom(\pi_{1}^{ab}(X),\mathbb{C}_{p}) \longrightarrow Hom(\pi_{1}^{ab}(X),\mathbb{C}_{p}^{*}). \]
  Composing the maps together, we get the desired map.    
    
  A morphism of rigid analytic spaces is called \textit{locally rigid analytic} if it is rigid analytic on an open subset. In the next proposition, we show that the morphism from the moduli space of the Higgs bundles to the moduli space of representations is locally rigid analytic with small conditions.
    
    \begin{proposition} (Proposition 4.1.5)
    One can enhance the map
    \[ \Gamma(X, \Omega_{X}^{1}) \otimes \mathbb{C}_{p}(-1) \longrightarrow Hom(\pi_{1}^{ab}(X), \mathbb{C}_{p}^{*})\] to the morphism of rigid analytic spaces
    \[ \Gamma(X,\Omega_{X}^{1}) \otimes \overline{\mathbb{Q}_{p}}(-1) \otimes \mathbb{G}_{a} \longrightarrow (\mathbb{G}_{m}^{2g})^{an} \]
    corresponding to the small representations of $Hom(\pi_{1}^{ab}(X), \mathbb{C}_{p}^{*})$.    
    \end{proposition}
    
    Finally, there is a bijection between Higgs bundles of degree 0 and representations in rank one case by \cite{Faltings 2005}. The main result of this article can be formulated as follows.
    
    \begin{theorem} (Theorem 4.2.1)
    Under some mild conditions in Theorem 3.2.6 and Proposition 4.1.5 and the assumption that \[f: (Pic^{0}_{X/\overline{\mathbb{Q}_{p}}})^{an} \times (\Gamma(X,\Omega_{X}^{1})\otimes \overline{\mathbb{Q}_{p}}(-1) \otimes \mathbb{G}_{a}) \longrightarrow (\mathbb{G}_{m}^{2g})^{an}\] is an affinoid morphism,  we deduce that the $p$-adic Simpson correspondence for line bundles
    \[ Hom(\pi_{1}^{ab}(X), \mathbb{C}_{p}^{*})_{small} \longrightarrow Pic^{0}_{X/\overline{\mathbb{Q}_{p}}}(\mathbb{C}_{p}) \times (\Gamma(X,\Omega_{X}^{1}) \otimes \mathbb{C}_{p}(-1))_{small} \] can be enhanced to the rigid analytic morphism defined on the open rigid analytic subspace $U \subset (\mathbb{G}_{m}^{2g})^{an}$  
    \[ (\mathbb{G}_{m}^{2g})^{an} \longrightarrow (Pic^{0}_{X/\overline{\mathbb{Q}_{p}}})^{an} \times (\Gamma(X, \Omega_{X}^{1}) \otimes \overline{\mathbb{Q}_{p}}(-1) \otimes \mathbb{G}_{a}). \]
    \end{theorem}

\begin{remark}
There is another proof of Theorem 1.0.3 in section 5 of \cite{Heuer}. To be more precise, for a smooth rigid analytic space $X$ over a perfectoid extension $K$ of $\mathbb{Q}_{p}$, they use the diamantine universal cover $\tilde{X} \longrightarrow X$ to construct the $p$-adic Simpson correspondence of rank one.
\end{remark}
    
\subsection*{Acknowledgements}
The work is part of my PhD thesis. I would like to thank my advisor, Mao Sheng, for his invaluable guidance. I am also grateful to the anonymous referee for their helpful comments on the first version of this article. The author was supported by NSFC under Grant Nos. 11721101.
    
\section{Preliminaries}
 In this section, we present some basic knowledge in rigid geometry. We shall use these results in the latter sections to prove our main theorems.
 \subsection{Maximum principle for affinoid $K$-algebra}
 Let $K$ be a field with a complete nonarchimedean absolute value that is nontrivial. For integers $n \geq 1$, let 
 \[ \mathbb{B}^{n}(\overline{K})=\{(x_{1},\dots,x_{n}) \in \overline{K}^{n}: |x_{i}| \leq 1\} \]
 be the unit ball in $\overline{K}^{n}$.
 
 \begin{definition}
 The $K$-algebra $T_{n}=K\langle\xi_{1},\dots,\xi_{n}\rangle$ of all formal power series
 \[ \sum \limits_{\nu \in \mathbb{N}^{n}} c_{\nu} \xi^{\nu} \in K[[\xi_{1},\dots,\xi_{n}]], \quad c_{\nu} \in K, \quad \lim \limits_{|\nu| \rightarrow \infty} |c_{\nu}|=0. \]
 i.e, converging on $\mathbb{B}^{n}(\overline{K})$, is called the \textit{Tate algebra} of restricted convergent power series.
 \end{definition}
 
 We define the \textit{Gauss norm} on $T_{n}$ by setting 
 \[ |f|=max|c_{\nu}|\quad for \quad f=\sum \limits_{\nu} c_{\nu} \xi^{\nu}. \]
 \begin{definition}
 A $K$-algebra $A$ is called an \textit{affinoid $K$-algebra} if there is an epimorphism of $K$-algebra $\alpha: T_{n} \longrightarrow A$ for some $n\in \mathbb{N}$.
 \end{definition}
 
 For elements $f\in A$, set
 \[ |f|_{sup}=\sup \limits_{x \in Max \,A} |f(x)| \]
 where $Max \,A$ is the spectrum of maximal ideals in $A$ and for any $x \in Max \,A$, write $f(x)$ for the residue class of $f$ in $A/x$.
 
 \begin{remark}
 In general case, $|\,\,|_{sup}$ is a $K$-algebra seminorm. Moreover, $|f|_{sup}=0$ is equivalent to that $f$ is nilpotent.
 \end{remark}

 Now we give an nonarchimedean analogue of the maximum principle in complex analysis.
 
 \begin{theorem}(Theorem 3.1.15 \cite{Bosch})
 For any affinoid $K$-algebra $A$ and for any $f \in A$, there exists a point $x \in Max \, A$ such that $|f(x)|=|f|_{sup}$. 
 \end{theorem}

 \subsection{Canonical topology on affinoid $K$-space}
 Let $A$ be an affinoid $K$-algebra, the elements of $A$ can be viewed as functions on $Max \, A$. More precisely, let us define $f(x)$ for $f \in A$ and $x \in Max \, A$ as the residue class of $f$ in $A/x$. Write $Sp\,A$ for the set $Max \, A$ together with its $K$-algebra of functions $A$ and call it the \textit{affinoid $K$-space} associated to $A$.  
 
 For an affinoid $K$-space $X=Sp \, A$, set
 \[ X(f,\varepsilon)= \{ x \in X: |f(x)| \leq \varepsilon \} \]
 for $f \in A$ and $\varepsilon \in \mathbb{R}_{>0}$.
 
 \begin{definition}
 For any affinoid $K$-space $X=Sp \, A$, the topology generated by all sets of type $X(f;\varepsilon)$ with $f \in A$ and $\varepsilon \in \mathbb{R}_{>0}$ is called the canonical topology of $X$. Define $X(f):=X(f;1)$.
 \end{definition}

 Roughly speaking, an \textit{affinoid subdomain} is an open subset with respect to canonical topology which has the structure of affinoid $K$-space (See Section 3.3 of \cite{Bosch}). A subset in $X$ of type 
 \[ X(f_{1},\dots,f_{r})=\{ x \in X: |f_{i}(x)| \leq 1 \} \]
 for functions $f_{1},\dots,f_{r} \in A$ is called \textit{Weierstrass domain} in $X$.
 
 Now we provide a basis for the canonical topology of affinoid $K$-space.
 
 \begin{proposition}(Proposition 3.3.5 \cite{Bosch})
 Let $X=Sp \, A$ be an affinoid $K$-space and $x \in X$ correspond to the maximal ideal $m_{x} \subset A$. Then the sets $X(f_{1},\dots,f_{r}):=X(f_{1}) \cap \dots \cap X(f_{r})$ for $f_{1},\dots,f_{r} \in m_{x}$ and variable $r$ form a basis of neighborhoods of $x$.
 \end{proposition}
 
 Finally, the next proposition shows that morphism between affinoid $K$-spaces is continuous with respect to the canonical topology.
 
 \begin{proposition}(Proposition 3.3.6 \cite{Bosch})
 Let $\varphi^{*}:A \longrightarrow B$ be a morphism of affinoid $K$-algebras, and let $\varphi:Sp\, B \longrightarrow Sp \, A$ be the associated morphism of affinoid $K$-spaces. Then for $f_{1},\dots,f_{r} \in A$, we have
 \[ \varphi^{-1}((Sp\,A)(f_{1},\dots,f_{r}))=(Sp\,B)(\varphi^{*}(f_{1}),\dots,\varphi^{*}(f_{r})).\]
 In particular, $\varphi$ is continuous with respect to the canonical topology.
 \end{proposition}

 \subsection{GAGA functor on rigid analytic space}
Let $X$ be an affinoid $K$-space. For any affinoid subdomain $U \subset X$ we denote the affinoid $K$-algebra corresponding to $U$ by $\mathcal{O}_{X}(U)$. $\mathcal{O}_{X}$ is a presheaf of affinoid $K$-algebras on the category of affinoid subdomains of $X$. Moreover, $\mathcal{O}_{X}$ is a sheaf by Tate's acyclicity Theorem as follows:

\begin{theorem}(Theorem 4.3.10 \cite{Bosch})
Let $X$ be an affinoid $K$-space and $\mathcal{U}$ is a finite covering of $X$ by affinoid subdomains. Then $\mathcal{U}$ is acyclic with respect to the presheaf $\mathcal{O}_{X}$ of affinoid functions on $X$.
\end{theorem}

\begin{definition}
A \textit{rigid analytic $K$-space} is a locally $G$-ringed space $(X,\mathcal{O}_{X})$ such that

(i) the Grothendieck topology of $X$ satisfies $(G_{0}),(G_{1})$ and $(G_{2})$ of 5.1/5 in \cite{Bosch}.

(ii) $X$ admits an admissible covering $(X_{i})_{i \in I}$ where $(X_{i}, \mathcal{O}_{X}|_{X_{i}})$ is an affinoid $K$-space for all $i \in I$.
\end{definition}

A \textit{morphism of rigid $K$-spaces} $(X,\mathcal{O}_{X}) \longrightarrow (Y, \mathcal{O}_{Y})$ is a morphism in the sense of locally $G$-ringed $K$-spaces. As usual, global rigid $K$-spaces can be constructed by gluing local ones.

Now there is a functor that associates any $K$-scheme $Z$ of locally of finite type to a rigid $K$-space $Z^{an}$, called the \textit{rigid analytification} of $Z$. For more details, see \cite{Bosch}.

\begin{proposition}(Proposition 5.4.4 \cite{Bosch})\label{proposition 2.10}
Every $K$-scheme $Z$ of locally finite type admits an analytification $Z^{an} \longrightarrow Z$.
Furthermore, the underlying map of sets identifies the points of $Z^{an}$ with the closed points of $Z$.
\end{proposition}

\begin{remark}
When $X$ is a variety over $K$, one can identify the points of $X^{an}$ with the points of $X$.
\end{remark}

\section{Vector bundle case}  
  
 \subsection{Lie algebra of Deninger-Werner's map $\alpha$}
   Suppose that $X$ is a curve with good reduction over $\overline{\mathbb{Q}_{p}}$ (i.e. there exists a smooth, proper, finitely presented model $\mathcal{X}$ of $X$ over $\overline{\mathbb{Z}_{p}}$) and denote the genus of $X$ by $g$. By \cite{Deninger and Werner 2005a} we obtain a continuous, Galois equivariant homomorphism
    \[ \alpha: Pic^{0}_{X/\overline{\mathbb{Q}_{p}}}(\mathbb{C}_{p}) \longrightarrow Hom_{c}(\pi_{1}^{ab}(X), \mathbb{C}_{p}^{*}). \]
    
   For the left-hand side, denote the Albanese variety of $X$ by $A$. Then the dual abelian variety $\widehat{A}$ of $A$ is isomorphic to $Pic^{0}_{X/\overline{\mathbb{Q}_{p}}}$ via the $\Theta$-polarization. Consequently $\widehat{A}(\mathbb{C}_{p}) =     
Pic^{0}_{X/\overline{\mathbb{Q}_{p}}}(\mathbb{C}_{p})$. For the right-hand side, first we claim that the Tate module $TA$ of the abelian variety $A$ can be identified with the maximal abelian quotient of the \'{e}tale fundamental group $\pi_{1}(X,x)$. In fact, let $f: X \rightarrow A$ be the embedding mapping $x$ to 0. Then $f$ induces the homomorphism 
\[f_{*}: \pi_{1}(X,x) \longrightarrow \pi_{1}(A,0). \]

Note that $\overline{\mathbb{Q}_{p}}$ is an algebraically closed field, Lang-Serre Theorem (refer to \cite{Gerard van der Geer and Ben Moonen}, p.159 Theorem (10.36))  asserts that any connected finite \'{e}tale cover of $A$ is also an abelian variety. It deduces that the \'{e}tale fundamental group of $A$ is canonically isomorphic to its Tate module.
Since $TA \simeq \prod \limits_{p \,prime}\mathbb{Z}_{p}^{2d}$ where $d=dim(A)$, the \'{e}tale fundamental group of $A$ is an abelian group. By universal property of Albanese variety, it follows that the Tate module $TA$ is isomorphic to $\pi_{1}^{ab}(X,x)$ as desired. Since the \'{e}tale fundamental group of $X$ is compact with respect to the profinite topology, we deduce that $\pi_{1}^{ab}(X)$ is compact. Note that $\mathbb{C}_{p}^{*}=p^{\mathbb{Q}} \times \textbf{o}^{*}$ and the fact that every element of the image of 1-dimensional representation of the compact group has norm 1, thus \[Hom(\pi_{1}^{ab}(X),\mathbb{C}_{p}^{*})=Hom(\pi_{1}^{ab}(X),\textbf{o}^{*})=Hom(TA,\textbf{o}^{*}).\]
   From the previous argument, the map $\alpha$ can be rephrased in terms of the Albanese variety $A$ of $X$ as a continuous, Galois equivariant homomorphism
   \begin{align}\label{(3.0.1)}
   \alpha: \widehat{A}(\mathbb{C}_{p}) \longrightarrow Hom_{c}(TA, \textbf{o}^{*}).
   \end{align}
By Corollary 14 of \cite{Deninger and Werner 2005b}, $\alpha$ is a $p$-adically analytic map of $p$-adic Lie groups.
 
   Let us now determine the Lie algebra map induced by $\alpha$.
   
   From \cite{Deninger and Werner 2005b}, we have the exact sequence of the logarithm on the Lie group $\widehat{A}(\mathbb{C}_{p})$
   \[ \xymatrix{0 \ar[r] & \widehat{A}(\mathbb{C}_{p})_{tors} \ar[r] & \widehat{A}(\mathbb{C}_{p}) \ar[r]^-{log} & Lie\widehat{A}(\mathbb{C}_{p})=H^{1}(A,\mathcal{O}) \otimes_{\overline{\mathbb{Q}_{p}}} \mathbb{C}_{p} \ar[r]& 0\\}. \]
   
   On the torsion subgroups, $\alpha$ is an isomorphism (Proposition 10 of \cite{Deninger and Werner 2005b}), so that we get the following commutative diagram with exact lines:
   
   \begin{align}\label{(3.0.2)}{ \xymatrix{0 \ar[r] & \widehat{A}(\mathbb{C}_{p})_{tors} \ar[d]^{\simeq} \ar[r] & \widehat{A}(\mathbb{C}_{p}) \ar[d]^{\alpha} \ar[r]^{log} & Lie\widehat{A}(\mathbb{C}_{p}) \ar[d]^{Lie\,\alpha}\ar[r]& 0\\
      0 \ar[r] & Hom_{c}(TA,\mu) \ar[r] & Hom_{c}(TA,\mathbb{C}_{p}^{*}) \ar[r] & Hom_{c}(TA,\mathbb{C}_{p}) \ar[r] & 0}}\end{align}
    
    Here $\mu$ is the subgroup of the roots of unity in $\mathbb{C}_{p}$.
         
 \subsection{Vector bundles give rise to representations}
   The goal of this section is to construct a rigid analytic morphism
    \[ \alpha^{an}: (Pic^{0}_{X/\overline{\mathbb{Q}_{p}}})^{an}  \longrightarrow
    (\mathbb{G}_{m}^{2g})^{an} \]
    such that the points of $\alpha^{an}$ coincide with the map $\alpha$ as (\ref{(3.0.1)}).
    
    Let us first consider the points of both sides. For the left-hand side, using Remark 2.3.4 we have $(Pic^{0}_{X/\overline{\mathbb{Q}_{p}}})^{an}(\mathbb{C}_{p})=Pic^{0}_{X/\overline{\mathbb{Q}_{p}}}(\mathbb{C}_{p})$. For the right-hand side, one can deduce that $(\mathbb{G}_{m}^{2g})^{an}(\textbf{o})=(\mathbb{G}_{m}^{2g})(\textbf{o})=(\textbf{o}^{*})^{2g}=Hom_{c}(TA,\textbf{o}^{*})$.
    
    To prove that $\alpha^{an}$ exists as a rigid analytic function, it suffices to construct it on an affinoid open subset. We break the proof into the following steps:
   \subsubsection{Step1: Reduction to affinoid neighbourhood}
   The morphism 
   \[\alpha^{an}: (Pic^{0}_{X/\overline{\mathbb{Q}_{p}}})^{an}  \longrightarrow (\mathbb{G}_{m}^{2g})^{an}\] 
   will map $0$ to $1$. On the point level, we have $\alpha^{an}(\mathbb{C}_{p})=\alpha(\mathbb{C}_{p})$. Fix an element $x_{0}\in (Pic^{0}_{X/\overline{\mathbb{Q}_{p}}})^{an}$ and define $y_{0}=\alpha(x_{0})$. Observe that both sides have rigid group structure. Consider the translation maps
   \[ T_{x_{0}}(x)=x+x_{0},\quad T_{y_{0}}(y)=y\cdot y_{0}\]  
   and the following diagram
   \begin{align}\label{(3.0.3)}{\xymatrixcolsep{5pc}\xymatrix{ U_{0} \ar[d]^{T_{x_{0}}} \ar[r]^{\alpha|_{U_{0}}} & U_{1} \ar[d]^{T_{y_{0}}}\\
   U_{x_{0}} \ar[r]^{\alpha|_{U_{x_{0}}}} & U_{y_{0}}} }\end{align}
   where $U_{1}$ is the neighborhood of $1 \in (\mathbb{G}_{m}^{2g})^{an}$ and $U_{0}=\alpha^{-1}(U_{1}),U_{y_{0}}=T_{y_{0}}(U_{1}),U_{x_{0}}=T_{x_{0}}(U_{0}) \cap \alpha^{-1}(U_{y_{0}})$ are open with respect to strong Grothendieck topology.
   
   For any $x \in U_{0}$, we have
   \[ (\alpha|_{U_{x_{0}}}\circ T_{x_{0}})(x)=\alpha(x+x_{0})=\alpha(x) \cdot \alpha(x_{0})=(T_{y_{0}}\circ \alpha|_{U_{0}})(x) \]
   Hence
   \[ \alpha|_{U_{x_{0}}}\circ T_{x_{0}}=T_{y_{0}}\circ \alpha|_{U_{0}}.\]
Namely, the above diagram is commutative.
 
   It suffices to prove that $\alpha|_{U_{0}}$ is rigid analytic. Then $\alpha|_{U_{x_{0}}}=T_{y_{0}}\circ \alpha|_{U_{0}}\circ T_{-x_{0}}$ is rigid analytic. Assume that $\alpha$ is continuous with respect to  strong Grothendieck topology. It yields that $\alpha^{an}$ is rigid analytic by gluing the rigid analytic morphisms. 
   
   The techinique step is to shrink $U_{0}$ to an open affinoid neighborhood in order to give the natural transformation of the sheaf of rigid analytic functions.
   
   Before shrinking, let us construct logarithm on rigid analytic space in the following diagram.
   
   \[ \xymatrixcolsep{5pc}\xymatrix{(Pic^{0}_{X/\overline{\mathbb{Q}_{p}}})^{an} \ar[d] \ar[r]^{log} & H^{1}(A,\mathcal{O}) \otimes \mathbb{G}_{a} \ar[d]^{l=(Lie\alpha)^{an}}\\
   (\mathbb{G}_{m}^{2g})^{an} \ar[r]^{log} & (\mathbb{G}_{a}^{2g})^{an}} \]
   
   Here we identify $Lie\, \textbf{o}^{*}$ with $\textbf{o}$ by means of the invariant differential $\frac{dT}{T}$ on $\mathbb{G}_{m}=Spec\, \textbf{o}[T,T^{-1}]$ and $Lie\, Pic^{0}_{X/\overline{\mathbb{Q}_{p}}}(\mathbb{C}_{p})=Lie\, \widehat{A}(\mathbb{C}_{p})=H^{1}(A,\mathcal{O})\otimes_{\overline{\mathbb{Q}_{p}}} \mathbb{C}_{p}$.   
   
   (\romannumeral 1) $log: (Pic^{0}_{X/\overline{\mathbb{Q}_{p}}})^{an} \longrightarrow H^{1}(A,\mathcal{O}) \otimes \mathbb{G}_{a}$  
   
   \vspace{0.2cm}
   We start with the definition of the rigid analytic $p$-divisible group introduced by Fargues \cite{Fargues}. A $p$-divisible group is defined as $p$-torsion, $p$-divisible such that $ker[p]$ is a finite group scheme by \cite{Bergamaschi} (line 7-9 of p.8). The motivation is to release the "$p$-torsion" condition of $p$-divisible group by replacing it with a condition of "topological $p$-torsion". Thus the rigid analytic $p$-divisible group can be viewed as a generalization of $p$-divisible group in rigid analytic setting.
   
   Let $K/\mathbb{Q}_{p}$ be a complete valued field.
   \begin{definition}
   A commutative rigid analytic $K$-group $G$ is said to be \textit{rigid analytic $p$-divisible group} if
   
   (i) (Topologically $p$-torsion\footnote{The notion "topologically $p$-torsion" is in the sense of Berkovich space. The Berkovich spectrum can be viewed as consisting of the data of prime ideal plus the extension of the norm to the residue field. Thus the Berkovich space $|G^{Berk}|$ has far more points than the rigid analytic space $|G^{an}|$.})
   
   For any $g \in G^{Berk}$, $\lim \limits_{n \rightarrow \infty} p^{n}g=0$ with respect to the topology of the Berkovich space $|G^{Berk}|$. Equivalently, if $U$ and $V$ are two affinoid neighborhoods of 0, then for $n \gg 0, p^{n}U \subset V$.
   
   (ii) The morphism $\times p: G \longrightarrow G$ is finite and surjective.
   \end{definition}
   Let $G$ be a rigid analytic $p$-divisible group. By Proposition 16 of \cite{Fargues}, there is a natural logarithm morphism of rigid analytic spaces
   \begin{align}\label{(3.2.2)} {log_{G}: G \longrightarrow Lie(G) \otimes \mathbb{G}_{a}}. \end{align}
  
   Assume that $(Pic^{0}_{X/\overline{\mathbb{Q}_{p}}})^{an}$ is topologically $p$-torsion. Note that $Pic^{0}_{X/\overline{\mathbb{Q}_{p}}}$ is an abelian variety, it follows that $[p]: Pic^{0}_{X/\overline{\mathbb{Q}_{p}}} \longrightarrow Pic^{0}_{X/\overline{\mathbb{Q}_{p}}}$ is an isogeny. The morphism $\times p: (Pic^{0}_{X/\overline{\mathbb{Q}_{p}}})^{an} \longrightarrow (Pic^{0}_{X/\overline{\mathbb{Q}_{p}}})^{an}$ is finite and surjective by Lemma 5 of \cite{Fargues}, we deduce that $(Pic^{0}_{X/\overline{\mathbb{Q}_{p}}})^{an}$ is a rigid analytic $p$-divisible group. Thus the logarithm defined in (\ref{(3.2.2)}) induces the logarithm of rigid analytic spaces
   \[ log: (Pic^{0}_{X/\overline{\mathbb{Q}_{p}}})^{an} \longrightarrow H^{1}(A,\mathcal{O}) \otimes \mathbb{G}_{a}. \]
   
   The $\mathbb{C}_{p}$-points of $H^{1}(A,\mathcal{O}) \otimes \mathbb{G}_{a}$ are $H^{1}(A,\mathcal{O}) \otimes_{\overline{\mathbb{Q}_{p}}} \mathbb{C}_{p}$. Consequently the $\mathbb{C}_{p}$-points of $log$ coincide with the logarithm on the Lie group $Pic^{0}_{X/\overline{\mathbb{Q}_{p}}}(\mathbb{C}_{p})$. Moreover, $log: (Pic^{0}_{X/\overline{\mathbb{Q}_{p}}})^{an} \longrightarrow H^{1}(A,\mathcal{O}) \otimes \mathbb{G}_{a}$ is a local isomorphism (i.e, there exist open subsets $U \subset (Pic^{0}_{X/\overline{\mathbb{Q}_{p}}})^{an}$ and $V \subset H^{1}(A,\mathcal{O}) \otimes \mathbb{G}_{a}$ such that $U \simeq V$).
   
   (\romannumeral 2) $log: (\mathbb{G}_{m}^{2g})^{an} \longrightarrow (\mathbb{G}_{a}^{2g})^{an}$ 
   
   \vspace{0.2cm}
   Note that topologically $p$-torsion elements of $\mathbb{G}_{m}^{an}$ is $\widehat{\mathbb{G}}_{m}^{an}:=\{ x\in \mathbb{G}_{m}^{an}:|x-1|<1 \}$ (Example 5(1) of \cite{Fargues} on p.17), we deduce that the logarithm $log: \mathbb{G}_{m}^{an} \longrightarrow \mathbb{G}_{a}^{an}$ is defined on $\widehat{\mathbb{G}}_{m}^{an}$. By (\ref{(3.2.2)}), the logarithm $log: (\mathbb{G}_{m}^{2g})^{an} \longrightarrow (\mathbb{G}_{a}^{2g})^{an}$ is defined on the open set of topologically $p$-torsion points of $(\mathbb{G}_{m}^{2g})^{an}$. 
       
   For the $\textbf{o}$-points, we have 
   \begin{align}\label{(3.0.4)} {log: (\mathbb{G}_{m}^{2g})^{an}(\textbf{o})=(\textbf{o}^{*})^{2g} \longrightarrow \textbf{o}^{2g}=(\mathbb{G}_{a}^{2g})^{an}(\textbf{o})}\end{align}
   
   Consider each component $log: \textbf{o}^{*} \longrightarrow \textbf{o}$, set $V_{0}=\{x \in \textbf{o}^{*}: |x-1|<p^{-\frac{1}{p-1}} \}$ and $W_{0}=\{ x \in \mathbb{C}_{p}: |x|<p^{-\frac{1}{p-1}} \}$. The logarithm provides an isomorphism
   \[ log: V_{0} \longrightarrow W_{0} \]
whose inverse is the exponential map. It follows that the logarithm sending $1$ to $0$ is a local homeomorphism. Taking product topology on both sides (\ref{(3.0.4)}), we have $log: (\mathbb{G}_{m}^{2g})^{an} \longrightarrow (\mathbb{G}_{a}^{2g})^{an}$ is a local homeomorphism(power series converges and hence continuous).   
   
   As a byproduct, we give the explicit expression of restriction of $\alpha$ on affinoid neighborhood.   
   \begin{definition}
   The rigid analytic space $(X,\mathcal{O}_{X})$ is \textit{reduced} if for any admissible open subset $U\subset X$, $\mathcal{O}_{X}(U)$ has no nilpotent elements.
   \end{definition}
   
   \begin{definition}
   A morphism $f: X\longrightarrow Y$ of rigid analytic spaces is called an \textit{affinoid morphism} if there is an admissible affinoid covering $\{V_{i}\}$ of $Y$ such that $f^{-1}(V_{i})$ is an affinoid space for each $i$.
   \end{definition}
      
   \begin{proposition}
   Let $K$ be a nonarchimedean field of char 0. If $X$ and $Y$ are rigid analytic spaces over $K$ and $X$ is reduced, then for any affinoid morphism $h: X \longrightarrow Y$, $h(\widehat{\overline{K}})=0$ implies $h=0$.   
   \end{proposition}
   \begin{proof}
   Let $Y=\cup _{i\in I} Sp B_{i}$ be an admissible affinoid covering. Since $h$ is an affinoid morphism, we obtain $h^{-1}(Sp B_{i})=Sp A_{i}$ for some affinoid algebra $A_{i}$. So we reduce to the affinoid case. 
   
   Let the corresponding morphism of affinoid algebras $B\longrightarrow A$ send $b$ to $a$. By maximum principle Theorem 2.1.4, there is a point $x\in Max\,A$ such that $|a|_{sup}=|a(x)|$. Let $\varphi:A\longrightarrow \mathbb{C}_{p}$ be the residue map of $x \in Max\,A$. By assumption, $h(\mathbb{C}_{p})=0$. From the    
commutative diagram 
   \[ \xymatrix{B \ar[d] \ar[r]^{0} & \textbf{o}\ar[d]\\
   A \ar[r]^{\varphi} & \mathbb{C}_{p}}\]    
one can deduce that $|a|_{sup}=|a(x)|=0$. Since $(Pic^{0}_{X/\overline{\mathbb{Q}_{p}}})^{an}$ is reduced, it follows that $A$ is reduced. Hence $a=0$ and $h=0$ as desired. 
   \end{proof}
   
   Since $X$ is a smooth projective curve, we have $Pic^{0}_{X/\overline{\mathbb{Q}_{p}}}$ is reduced. It follows that $(Pic^{0}_{X/\overline{\mathbb{Q}_{p}}})^{an}$ is reduced. The next corollary gives the explicit expression of $\alpha|_{U_{0}}$.
   
   \begin{corollary}
   Suppose that $\alpha|_{U_{0}}-exp \circ l \circ log$ is an affinoid morphism, we have $\alpha|_{U_{0}}=exp \circ l \circ log$.
   \end{corollary}
   \begin{proof}
   From the commutative diagram (\ref{(3.0.2)}), we have \[\alpha|_{U_{0}}(\mathbb{C}_{p})=exp \circ l \circ log(\mathbb{C}_{p}).\] By Proposition 3.2.4, it follows that $\alpha|_{U_{0}}=exp \circ l \circ log$.
   \end{proof}

   \vspace{0.2cm}
   Now we turn to the shrinking step.
   
   The logarithm $log: (\mathbb{G}_{m}^{2g})^{an} \longrightarrow (\mathbb{G}_{a}^{2g})^{an}$ induces the isomorphism 
   \[ log: V_{0}=\{x \in \textbf{o}^{*}: |x-1|<r_{p}:=p^{-\frac{1}{p-1}}\}  \longrightarrow W_{0}=\{ x \in \mathbb{C}_{p}: |x|<p^{-\frac{1}{p-1}}\}. \]
   
   First we shrink $V_{0}$ and $W_{0}$ to affinoid space. Choose $\varepsilon \in \mathbb{R}_{>0}$ such that $r_{p}^{\prime}=r_{p}-\varepsilon \in |\mathbb{C}_{p}^{\times}|=|p|^{\mathbb{Q}}$, and thus $V_{0}^{\prime}:=\{x\in \textbf{o}^{*}: |x-1| \leq r_{p}^{\prime} \}\subset V_{0}$ is an affinoid space. Set $W_{0}^{\prime}:=\{ x\in \mathbb{C}_{p}: |x|\leq r_{p}^{\prime}\}\subset W_{0}$ and $W_{0}^{\prime}$ is again an affinoid space. Let $x\in V_{0}$ and $y=x-1$. It follows that $|log(x)|=|log(1+y)|=|y|\,\, for\, |y|<r_{p}$. Hence $log(V_{0}^{\prime})=W_{0}^{\prime}$ and the logarithm provides the isomorphism
   \[ log: V_{0}^{\prime} \longrightarrow W_{0}^{\prime}. \]
    
Now consider the logarithm $log: (Pic^{0}_{X/\overline{\mathbb{Q}_{p}}})^{an} \longrightarrow H^{1}(A,\mathcal{O}) \otimes \mathbb{G}_{a}$. Find an affinoid neighborhood $Sp A\subset H^{1}(A,\mathcal{O}) \otimes \mathbb{G}_{a}$ of $0$ and shrink the inverse image $log^{-1}(Sp A)$ to an affinoid neighborhood $Sp B$ of $0\in (Pic^{0}_{X/\overline{\mathbb{Q}_{p}}})^{an}$. For the morphism of affinoid spaces $\varphi=log|_{Sp B}: Sp B \longrightarrow Sp A$, let it correspond to a morphism of affinoid algebras $\varphi^{*}: A \longrightarrow B$. Since $log$ is a local isomorphism, there are admissble open subsets $V_{1}\subset Sp B$ and $W_{1} \subset Sp A$ such that $V_{1} \simeq W_{1}$. Let $W_{1}^{\prime}:=W_{1}\cap l^{-1}(W_{0}^{\prime})$ and $V_{1}^{\prime}:=log^{-1}(W_{1}\cap l^{-1}(W_{0}^{\prime}))$, then $V_{1}^{\prime} \simeq W_{1}^{\prime}$. Note that the canonical topology is finer than the Grothendieck topology on affinoid spaces, applying Proposition 2.2.3, there exist $f_{1},\dots,f_{r} \in A$ such that $(Sp A)(f_{1},\dots,f_{r}) \subset W_{1}^{\prime}$. Furthermore, we have
\[ \varphi^{-1}((Sp A)(f_{1},\dots,f_{r}))=(Sp B)(\varphi^{*}(f_{1}),\dots,\varphi^{*}(f_{r}))\]
Observe that $(Sp A)(f_{1},\dots,f_{r})$ and $(Sp B)(\varphi^{*}(f_{1}),\dots,\varphi^{*}(f_{r}))$ are Weierstrass domains, thus open affinoids. Since $\varphi$ is continuous with respect to Grothendieck topology (Proposition 3.1.6 of \cite{Bosch}), it follows that there is a homeomorphism between $(Sp B)(\varphi^{*}(f_{1}),\dots,\varphi^{*}(f_{r}))$ and $(Sp A)(f_{1},\dots,f_{r})$. We finish the shrinking step.

  Let $U_{1}=V_{0}^{\prime}\subset (\mathbb{G}_{m}^{2g})^{an}$ in the diagram (\ref{(3.0.3)}). Assume that \[U_{0}=(Sp B)(\varphi^{*}(f_{1}),\dots,\varphi^{*}(f_{r}))\] after shrinking. Note that $\alpha^{an}|_{U_{0}}: U_{0} \longrightarrow U_{1}$ is a morphism of affinoid spaces and the logarithm provides the local homeomorphism, the morphism \[l: H^{1}(A,\mathcal{O})\otimes \mathbb{G}_{a} \longrightarrow (\mathbb{G}_{a}^{2g})^{an}\] gives the morphism of corresponding affinoid algebras from the commutative diagram (\ref{(3.0.2)}) on the point level. Using Lemma 5.3.3 of \cite{Bosch}, we deduce that $\alpha^{an}|_{U_{0}}$ is rigid analytic as desired. 
  
  Now we have proved the following theorem.   
  
  \begin{theorem}
  Assume that $(Pic^{0}_{X/\overline{\mathbb{Q}_{p}}})^{an}$ is topologically $p$-torsion. Then we may enhance the set-theoretical map
  \[ \alpha: Pic^{0}_{X/\overline{\mathbb{Q}_{p}}}(\mathbb{C}_{p}) \longrightarrow Hom_{c}(\pi_{1}^{ab}(X), \mathbb{C}_{p}^{*}) \] to the rigid analytic morphism
  \[ \alpha^{an}: (Pic^{0}_{X/\overline{\mathbb{Q}_{p}}})^{an}  \longrightarrow
    (\mathbb{G}_{m}^{2g})^{an} \]  
  \end{theorem}   
   
\section{Higgs bundle case}
 In this section, we prove that the morphism from the moduli space of the Higgs bundles to the moduli space of representations is locally rigid analytic with small conditions.
 
 Let us begin with the construction of expontential map from $\mathbb{C}_{p}$ to $\mathbb{C}_{p}^{*}$. Observe that the series $\sum_{k \geq 0} \frac{x^{k}}{k!}$ converges precisely when $|x|<r_{p}=|p|^{\frac{1}{p-1}}$. By analogy with the classical case, we write 
 \[ exp(x)=\sum_{k \geq 0} \frac{x^{k}}{k!} \]
for the sums of these series whenever they converge. We deduce the property that $exp(x+y)=exp(x)\cdot exp(y)$ for $|x|<r_{p}$ and $|y|<r_{p}$. It is natural to try to construct a homomorphism
\[ f: \mathbb{C}_{p} \longrightarrow \mathbb{C}_{p}^{*} \]
extending the exponential defined above by a series expansion. If such an extension exists, when $x\in \mathbb{C}_{p}$ we can choose a high power $p^{n}$ of $p$ so that 
$p^{n}x \in B_{<r_{p}}$ and then
\[ f(x)^{p^{n}}=f(p^{n}x)=exp(p^{n}x). \]
In other words, $f(x)$ has to be a $p^{n}$th root of $exp(p^{n}x)$ in the algebraically closed field $\mathbb{C}_{p}$. This can be done in a coherent way, thus furnishing a continuation of the expontential homomorphism.

 Denote the maximal ideal $\{x \in \mathbb{C}_{p}: |x|<1 \}$ of the local ring $\{ x \in \mathbb{C}_{p}: |x| \leq 1 \}$ by $M_{p}$. 
 \begin{proposition}(Proposition on p.258 \cite{Robert})
 There is a continuous homomorphism $Exp: \mathbb{C}_{p} \longrightarrow 1+M_{p}$ extending the exponential mapping, originally defined only on the ball $B_{<r_{p}} \subset \mathbb{C}_{p}$.
 \end{proposition}
 
 \begin{remark}
 (1) The definition of $Exp$ is indeed not canonical. The possibility of extending the exponential to the whole $\mathbb{C}_{p}$ had been shown by \cite{Durix}.
 
 (2) The diffference between $exp$ and $Exp$ is simply the domain of definition, which can be summarized in the following commutative diagram 
 \[\xymatrix{
 B_{<r_{p}} \ar@{^{(}->}[d] \ar[r]^{exp} & 1+B_{<r_{p}}\ar@{^{(}->}[d]\\
 \mathbb{C}_{p} \ar[r]^{Exp} & 1+M_{p}
 }\]
 \end{remark}

 \subsection{Higgs field gives rise to representations}
  Now we present the construction of the map
  \[ \Gamma(X,\Omega_{X}^{1}) \otimes_{\overline{\mathbb{Q}_{p}}} \mathbb{C}_{p}(-1) \longrightarrow Hom(\pi_{1}^{ab}(X),\mathbb{C}_{p}^{*}) \]
  
  Before the construction, we remark that there is no Galois action on the Tate twist $\Gamma(X,\Omega_{X}^{1}) \otimes_{\overline{\mathbb{Q}_{p}}} \mathbb{C}_{p}(-1)$. The above map is the exponential of a $\mathbb{C}_{p}$-linear map into the Lie algebra of the group of representations. More precisely, the $\mathbb{C}_{p}$-linear map is 
  \[ \Gamma(X,\Omega_{X}^{1}) \otimes_{\overline{\mathbb{Q}_{p}}} \mathbb{C}_{p}(-1) \longrightarrow (\Gamma(X,\Omega_{X}^{1}) \otimes \mathbb{C}_{p}(-1)) \oplus (H^{1}(X,\mathcal{O}_{X}) \otimes \mathbb{C}_{p}) \] 
where the first component is the identity, and the second one depends on the choice of lifting of $X$ to $A_{2}(V)$\footnote{The definition of $A_{2}(V)$ is defined as follows: Let $V$ be a complete DVR with perfect residue field $k$ of char $p>0$ and fraction field $K$ of char $0$, then the Fontaine's ring $A_{inf}(V)$ surjects onto the $p$-adic completion $\widehat{\overline{V}}$ and the kernel is a principal ideal with the generator $\xi$. We define $A_{2}(V)$ to be $A_{inf}(V)/(\xi^{2})$.} by \cite{Faltings 2005} (line 16-17 of p.861). Using the Hodge-Tate decomposition, we have the isomorphism
\[ H^{1}_{\acute{e}t}(X,\overline{\mathbb{Q}_{p}}) \otimes_{\overline{\mathbb{Q}_{p}}} \mathbb{C}_{p} \simeq  (\Gamma(X,\Omega_{X}^{1}) \otimes \mathbb{C}_{p}(-1)) \oplus (H^{1}(X,\mathcal{O}_{X}) \otimes \mathbb{C}_{p}).\]
Note that $Hom(\pi_{1}^{ab}(X),\mathbb{C}_{p}) \simeq H^{1}_{\acute{e}t}(X,\overline{\mathbb{Q}_{p}}) \otimes_{\overline{\mathbb{Q}_{p}}} \mathbb{C}_{p}$, then the exponential map from $\mathbb{C}_{p}$ to $\mathbb{C}_{p}^{*}$ induces the map
  \[ Hom(\pi_{1}^{ab}(X),\mathbb{C}_{p}) \longrightarrow Hom(\pi_{1}^{ab}(X),\mathbb{C}_{p}^{*}). \]
  Composing the maps together, we get the desired map.
  
  \subsubsection{Exp is not rigid analytic}
  \begin{lemma}
  The series $\sqrt[p]{1+x}=\sum_{k=0}^{\infty} \binom{\frac{1}{p}}{k}x^{k}$ converges when $|x|<|p|^{\frac{p}{p-1}}$.
  \end{lemma}
  \begin{proof}
  By strong triangle inequality, convergence of the series $\sum_{k=0}^{\infty} \binom{\frac{1}{p}}{k}x^{k}$ is equivalent to the condition $|\binom{\frac{1}{p}}{k}x^{k}| \rightarrow 0$. Note that the order $ord_{p}(k!)=\frac{k-S_{p}(k)}{p-1}$ where $S_{p}(k)$ is the sum of digits with respect to $p$-adic expansion of $k$, one has
  \[ |\binom{\frac{1}{p}}{k}x^{k}|=|p|^{k(ord_{p}(x)-\frac{p}{p-1})+\frac{S_{p}(k)}{p-1}}.\]
Since $S_{p}(k)=1$ when $k=p^{j}$ is a power of $p$, we have
\[ |\binom{\frac{1}{p}}{k}x^{k}| \rightarrow 0 \Longleftrightarrow k(ord_{p}(x)-\frac{p}{p-1}) \rightarrow \infty, \]
and this happens precisely when $ord_{p}(x)-\frac{p}{p-1} >0$, namely when 
\[ ord_{p}(x) >\frac{p}{p-1}, \quad |x|<|p|^{\frac{p}{p-1}}\]
as asserted.
  \end{proof}

  \begin{proposition}
  The exponential map $Exp$ is not rigid analytic for every choice of $Exp$.
  \end{proposition}
  \begin{proof}
  The main idea of the proof is to illustrate that we can not extend $exp$ to $Exp$ rigid analytically by progressively increasing balls. 
  
  For any $x \in \mathbb{C}_{p}$, there is an integer $n$ such that $p^{n}x \in B_{<r_{p}}(0)$. Thus $Exp$ is defined on the ball $B_{<|p|^{-n}r_{p}}(0)$, where the integer $n$ depends on $x$. We make the construction inductively as follows.
   
  Since $Exp$ is a group homomorphism, we obtain $Exp(px)=Exp(x)^{p}$. Thus $Exp(x)=\sqrt[p]{Exp(px)}$. Set $exp=(Exp)_{0}$. Construct $(Exp)_{1}$ as the compositions of the maps
  \[ \xymatrix{B_{|p|^{-1}r_{p}} \ar[r]^{\times p} &B_{<r_{p}} \ar[r]^{exp}&1+B_{<r_{p}} \ar[r]^{\sqrt[p]{1+x}} &1+M_{p}}.\]
  Inductively, we can construct $(Exp)_{n}$ on the ball $B_{<|p|^{-n}r_{p}}(0)$ from $(Exp)_{n-1}$. We claim that $(Exp)_{1}$ is not rigid analytic. For $y \in B_{<|p|^{-1}r_{p}}(0)$, let $exp(py)=1+x$ for some $x \in B_{<r_{p}}(0)$. It implies that
  \[ |x|=|log(1+x)|=|py|<r_{p}.\]
But by the Lemma 4.1.1, $\sqrt[p]{1+x}$ converges for $|x|<|p|^{\frac{p}{p-1}}<r_{p}$. Thus $(Exp)_{1}$ is not in a certain Tate algebra $\mathbb{C}_{p} \langle x \rangle$. 
Note that the exponential map $Exp$ is locally analytic. In fact, fix $x_{0} \in \mathbb{C}_{p}$. For $x \in B_{<r_{p}}(x_{0})$, we obtain the local expansion 
\[Exp(x)=Exp(x_{0})exp(x-x_{0})=Exp(x_{0})\sum_{k=0}^{\infty} \frac{(x-x_{0})^{k}}{k!}\] as wanted. Thus for any extension $\varphi$ of $exp$ to a continuous group homomorphism $\mathbb{C}_{p} \longrightarrow (1+M_{p})^{\times}$, $\varphi$ is not rigid analytic.
It follows that $(Exp)_{1}$ is not rigid analytic as claimed. 

Construct $(Exp)_{2}$ as the compositions of the maps
  \begin{align}\label{(4.1.1)}
  {\xymatrix{B_{|p|^{-2}r_{p}} \ar[r]^{\times p} &B_{|p|^{-1}r_{p}} \ar[r]^{(Exp)_{1}} &1+M_{p} \ar[r]^{\sqrt[p]{1+x}} &\mathbb{C}_{p}^{*}}.}
  \end{align}
Since $(Exp)_{1}$ is not in a certain Tate algebra, we have $(Exp)_{2}$ is not in a certain Tate algebra by the above commutative diagram (\ref{(4.1.1)}). 
  
  By the construction of $(Exp)_{n}$, we can not extend $exp$ to $Exp$ rigid analytically. Hence $Exp$ is not rigid analytic.
  \end{proof}

  \subsubsection{Small condition}
  
  Since the exponential map $Exp$ is not rigid analytic as shown in Proposition 4.1.2, we would like to add small conditions, which is the restriction of $Exp$ on the open ball $\{x \in \mathbb{C}_{p}: |x|<r_{p} \}$. Under the small condition, the exponential map is rigid analytic.
  
  \begin{definition}
  The Higgs bundle $(E, \theta)$ on $X$ is called \textit{small} if $p^{\alpha} | \,\theta$ for some $\alpha>\frac{1}{p-1}$. 
  \end{definition}
  
  \begin{definition}
  The representation $\rho \in Hom(\pi_{1}^{ab}(X),\textbf{o}^{*})$ is called \textit{small} if $\rho(\gamma_{i}) \equiv 1 \,(mod\, p^{2\beta})$, where $\gamma_{1},\dots, \gamma_{2g}$ are any set of generators of $\pi_{1}^{ab}(X)$.   \end{definition}

 By Faltings' construction in \cite{Faltings 2005}, there is a one-to-one correspondence between small rank one Higgs bundles of degree $0$ and small representations. In the rest of the article, we denote small Higgs field by 
$(\Gamma(X,\Omega_{X}^{1}) \otimes \mathbb{C}_{p}(-1))_{small}$ and denote small representations by $Hom(\pi_{1}^{ab}(X), \textbf{o}^{*})_{small}$. The difference are the following: for small Higgs fields, consider the multiplication on Higgs fields by $p^{\alpha}$ for some $\alpha>\frac{1}{p-1}$. For small representations,  denote by $E^{2g}$ by the $2g$-dimensional closed unit polydisc. Let $Y=E^{2g} \cap (\mathbb{G}_{m}^{2g})^{an}$, note that the elements of $p^{2\beta}\textbf{o}$ are topologically nilpotent, it follows that the surjective map $\textbf{o} \longrightarrow \textbf{o}/p^{2\beta}\textbf{o}$ induces the reduction map $\pi: Y \longrightarrow \widetilde{Y}$ (for the definition of the reduction map, we refer to Section 7.1.5 of \cite{BGR}). Then $U=\pi^{-1}(\{1\})$ is a tube\footnote{The definition of the tube is defined as follows: Let $X$ be a rigid analytic space and $x$ is the point of $X$. Denote the image of $x$ under the reduction map $red: X\longrightarrow  \widetilde{X}$ by $\widetilde{x}$. Define the tube of $\widetilde{x}$ in $X$ to be $red^{-1}(\widetilde{x})$.} of 1 in $Y\subset (\mathbb{G}_{m}^{2g})^{an}$ and $(U,\mathcal{O}_{U})$ is a rigid analytic subspace of $(\mathbb{G}_{m}^{2g})^{an}$.
The $\textbf{o}$-points of $U$, denote by $U(\textbf{o})$, is the small representations of $Hom(\pi_{1}^{ab}(X), \textbf{o}^{*})$. It follows that the $p$-adic Simpson correspondence is defined on the open subspace $(U,\mathcal{O}_{U})$ under small conditions. In the following, we fix the parameters $\alpha$ and $\beta$ as defined in Definition 4.1.3 and Definition 4.1.4 respectively. The relationship between $\alpha$ and $\beta$ is $2\beta=\alpha+\frac{1}{p-1}$.
   
  \begin{proposition}
  One can enhance the map
   \[ \Gamma(X, \Omega_{X}^{1}) \otimes \mathbb{C}_{p}(-1) \longrightarrow Hom(\pi_{1}^{ab}(X), \mathbb{C}_{p}^{*})\] to the morphism of rigid analytic spaces
   \[ \Gamma(X,\Omega_{X}^{1}) \otimes \overline{\mathbb{Q}_{p}}(-1) \otimes \mathbb{G}_{a} \longrightarrow (\mathbb{G}_{m}^{2g})^{an}. \]
  corresponding to the small representations of $Hom(\pi_{1}^{ab}(X), \mathbb{C}_{p}^{*})$.   
  \end{proposition}
  \begin{proof}
  For a small Higgs field $\theta \in \Gamma(X, \Omega_{X}^{1}) \otimes \mathbb{C}_{p}(-1)$, we have $p^{\alpha} | \,\theta$ for some $\alpha>\frac{1}{p-1}$. Denote the $\mathbb{C}_{p}$-linear map
  \[ \Gamma(X,\Omega_{X}^{1}) \otimes_{\overline{\mathbb{Q}_{p}}} \mathbb{C}_{p}(-1) \longrightarrow (\Gamma(X,\Omega_{X}^{1}) \otimes \mathbb{C}_{p}(-1)) \oplus (H^{1}(X,\mathcal{O}_{X}) \otimes \mathbb{C}_{p}) \]
  by $\beta$. Thus $p^{\alpha}|\,\beta(\theta)$ and $|\beta(\theta)|<r_{p}=|p|^{\frac{1}{p-1}}$. Since the difference between $exp$ and $Exp$ is simply domain of definition, it follows that $Exp(\beta(\theta))=exp(\beta(\theta))$. Note that $\beta$ is linear and defined by polynomials, thus algebraic, then $\beta$ is rigid analytic via the rigid analytification functor. Hence $exp\circ \beta$ is rigid analytic as desired. 
  \end{proof}
  
  \begin{remark}
  Because of the smallness condition, the non-uniqueness of the extension $Exp$ is not a problem in the proof of Proposition 4.1.5.
  \end{remark}
   
 \subsection{Conclusion}
  As our main result, we show that the $p$-adic Simpson correspondence for line bundles is rigid analytic under suitable conditions.
  
  By Theorem 3.2.6 and Proposition 4.1.5, it turns out that
  \[ f: (Pic^{0}_{X/\overline{\mathbb{Q}_{p}}})^{an} \times (\Gamma(X,\Omega_{X}^{1})\otimes \overline{\mathbb{Q}_{p}}(-1) \otimes \mathbb{G}_{a}) \longrightarrow U \subset (\mathbb{G}_{m}^{2g})^{an} \]
  is rigid analytic. On the point level, the map
  \[ Pic^{0}_{X/\overline{\mathbb{Q}_{p}}}(\mathbb{C}_{p}) \times (\Gamma(X,\Omega_{X}^{1}) \otimes \mathbb{C}_{p}(-1))_{small} \longrightarrow Hom(\pi_{1}^{ab}(X),\mathbb{C}_{p}^{*})_{small} \] mapping $(L,\theta)$ to $\rho$ is a homeomorphism by \cite{Faltings 2005}. 
  
  \begin{theorem}
  Under some mild conditions in Theorem 3.2.6 and Proposition 4.1.5 and the assumption that \[f: (Pic^{0}_{X/\overline{\mathbb{Q}_{p}}})^{an} \times (\Gamma(X,\Omega_{X}^{1})\otimes \overline{\mathbb{Q}_{p}}(-1) \otimes \mathbb{G}_{a}) \longrightarrow (\mathbb{G}_{m}^{2g})^{an}\] is an affinoid morphism,  we deduce that the $p$-adic Simpson correspondence for line bundles
    \[ Hom(\pi_{1}^{ab}(X), \mathbb{C}_{p}^{*})_{small} \longrightarrow Pic^{0}_{X/\overline{\mathbb{Q}_{p}}}(\mathbb{C}_{p}) \times (\Gamma(X,\Omega_{X}^{1}) \otimes \mathbb{C}_{p}(-1))_{small} \] can be enhanced to the rigid analytic morphism defined on the open rigid analytic subspace $U \subset (\mathbb{G}_{m}^{2g})^{an}$  
    \[ (\mathbb{G}_{m}^{2g})^{an} \longrightarrow (Pic^{0}_{X/\overline{\mathbb{Q}_{p}}})^{an} \times (\Gamma(X, \Omega_{X}^{1}) \otimes \overline{\mathbb{Q}_{p}}(-1) \otimes \mathbb{G}_{a}). \]
  \end{theorem}
  \begin{proof}
  By the previous argument, $f$ is rigid analytic. It suffices to prove that
  \[ f^{-1}: (\mathbb{G}_{m}^{2g})^{an} \longrightarrow (Pic^{0}_{X/\overline{\mathbb{Q}_{p}}})^{an} \times (\Gamma(X,\Omega_{X}^{1}) \otimes \overline{\mathbb{Q}_{p}}(-1) \otimes \mathbb{G}_{a}) \]
  is rigid analytic. Since $f$ is an affinoid morphism, one can reduce to the affinoid case. Set $X=Sp B$ and $Y=Sp A$. Let $f: (X,\mathcal{O}_{X}) \longrightarrow (Y,\mathcal{O}_{Y})$ be the morphism of locally $G$-ringed spaces. Then there is a bijection between morphisms of affinoid spaces $X \longrightarrow Y$ and morphisms of locally $G$-ringed spaces $(X,\mathcal{O}_{X}) \longrightarrow (Y,\mathcal{O}_{Y})$. One can view $f:X \longrightarrow Y$ as the morphism of affinoid spaces. Note that the category of affinoid spaces is the opposite category of affinoid algebras. Let $\varphi: A \longrightarrow B$ be the associated morphism of affinoid algebras. It follows that $\varphi^{-1}: B \longrightarrow A$ is the morphism of affinoid algebras. Thus $f^{-1}$ is rigid analytic. 
  \end{proof}


\begin{thebibliography}{X-X00}
\addcontentsline{toc}{chapter}{Bibliography}
\bibliographystyle{plain}
\bibitem{Faltings 2005}G. Faltings, A $p$-adic Simpson correspondence, \textit{Advance in Mathematics} \textbf{198(2)} (2005), 847-862.
\bibitem{Simpson 1992}C.T. Simpson, Higgs bundles and local systems, \textit{Publ. Math. IH\'{E}S} \textbf{75} (1992), 5-95.
\bibitem{Simpson 1994}C.T. Simpson, Moduli of representations of the fundamental group of a smooth projective variety \uppercase\expandafter{\romannumeral 1}, \textit{Publ. Math. IH\'{E}S} \textbf{79} (1994), 47-129.
\bibitem{Simpson}C.T. Simpson, Moduli of representations of the fundamental group of a smooth projective variety   \uppercase\expandafter{\romannumeral 2}, \textit{Publ. Math. IH\'{E}S} \textbf{80} (1994), 5-79.
\bibitem{Deninger and Werner 2005a}C. Deninger and A. Werner, Vector bundles on $p$-adic curves and parallel transport, \textit{Ann. Sci. \'{E}c. Norm. Sup\'{e}r.} (4) \textbf{38} (2005), 553-597.
\bibitem{Heuer}Ben Heuer, Diamantine Picard functors of rigid spaces. arxiv preprint 2013.16557v1, 2021.
\bibitem{Bosch}S. Bosch, \textit{Lectures on formal and rigid geometry}, Lecture Notes in Math. 2105, Springer, 2014.
\bibitem{Gerard van der Geer and Ben Moonen}Gerard van der Geer and Ben Moonen, \textit{Abelian varieties}, Book in preparation 71, 2007. Available at http://page.mi.fu-berlin.de/elenalavanda/bmoonen.pdf.
\bibitem{Deninger and Werner 2005b}C. Deninger and A. Werner, \textit{Line bundles and $p$-adic characters}, Number Fields and Function Fields--Two Parallel Worlds, Progress in Mathematics, vol. 239, Birkhauser Boston, 2005.  
\bibitem{Fargues}L. Fargues, Groupes analytiques rigides $p$-divisibles, \textit{Mathematische Annalen} \textbf{374} (2019), 723-791.
\bibitem{Bergamaschi}F. Bergamaschi, \textit{Abelian varieties and $p$-divisible groups.} Available at https://www.math.mcgill.ca/darmon/courses/12-13/sem-iovita/Bergamaschi-Seminarelovita.pdf.
\bibitem{Robert}A. M. Robert, \textit{A course in $p$-adic analysis}, Graduate Text in Math. 198, Springer-Verlag, New York, 2000.
\bibitem{Durix}Marie-Claude Durix, Prolongement de la function exponentielle en dehors de son disque de convergence, \textit{S\'{e}minaire Delange-Pisot-Poitou}:1966/67, \textit{Th\'{e}orie des Nombres}, Fasc.\textbf{1}, Exp.\textbf{1}, 12pp.  (Secr\'{e}tariat math\'{e}matique, Paris, 1968).
\bibitem{BGR}S. Bosch, U.G\"{u}ntzer and R. Remmert, \textit{Non-Archimedean analysis, a systematic approach to rigid analytic geometry}, volume 261 of Grundlehren der Mathematischen Wissenschaften, Springer-Verlag, Berlin, 1984.
\end{thebibliography}
\end{document}